%% file: paper.tex
\documentclass[12pt]{article}
\usepackage{fullpage,graphicx,psfrag,amsmath,amsfonts,verbatim,url,amssymb,amsthm}
\usepackage[small,bf]{caption}
\usepackage[font=footnotesize]{subcaption}
\usepackage{algorithm}
\usepackage[noend]{algpseudocode}
\usepackage[T1]{fontenc}
\usepackage{tikz}
\usetikzlibrary{positioning,calc}
\definecolor {processblue}{cmyk}{0.96,0,0,0}
\usepackage{graphicx}
\usepackage{color}
\usepackage[ruled,vlined, linesnumbered, algo2e]{algorithm2e}
\usepackage{algpseudocode}
\usepackage{booktabs}
\usepackage[para,online,flushleft]{threeparttable}
\usepackage[autostyle]{csquotes}
\usepackage{supertabular}
\usepackage{algorithmicx}
\usepackage{amsmath}
\usepackage{multirow}
\usepackage{pifont}
\usepackage{lscape}
\usepackage{amsthm}

\usepackage[singlelinecheck=off]{caption}
\usepackage{hyperref}
\hypersetup{
	colorlinks   = true,
	citecolor   = blue
}
\theoremstyle{definition}

\algrenewcommand\alglinenumber[1]{
    {\sf\footnotesize\addfontfeatures{Colour=888888,Numbers=Monospaced}#1}}
\algrenewcommand\algorithmicrequire{\textbf{Precondition:}}
\algrenewcommand\algorithmicensure{\textbf{Postcondition:}}
\mathchardef\mhyphen="2D
\renewcommand{\algorithmicrequire}{\textbf{Input:}}
\newcommand*\samethanks[1][\value{footnote}]{\footnotemark[#1]}
\newcommand{\quotes}[1]{``#1''}
\newtheorem{definition}{Definition}
\newtheorem{remark}{Remark}

\input defs.tex

\title{Anti-lopsided Algorithm for Large-scale Non-negative Least Squares}
\author{Duy-Khuong Nguyen
		~\thanks{Japan Advanced Institute of Science and Technology, Japan}
		~\thanks{University of Engineering and Technology, Vietnam National University, Hanoi, Vietnam}
	\and Tu-Bao Ho
		~\samethanks[1]
		~\thanks{John von Neumann Institute, Vietnam National University, Ho Chi Minh City, Vietnam}}

\begin{document}
\maketitle

\begin{abstract}
Non-negative least squares~(NNLS)  is one of the most fundamental problems in numeric analysis. It has been widely used in scientific computation and data modeling. 
When applying to huge and complex data in big data analytics, the limitations of NNLS algorithm speed and accuracy remain as key challenges. In this paper, we develop a fast and robust anti-lopsided algorithm for NNLS with high accuracy that is totally based on the first order methods. 
The main idea of the proposed algorithm is to transform the original NNLS into an equivalent non-negative quadratic programming, which significantly reduces the scaling problem of variables to discover more effective gradient directions.  The proposed algorithm reaches high accuracy and fast speed  with linear convergence $(1 - \frac{1}{2||Q||_2})^k$ in the subspace of passive variables where $\sqrt{n} \leq ||Q||_2 \leq n$, and $n$ is the dimension size of solutions. The experiments on large matrices clearly show the high performance of the proposed algorithm in comparison to the state-of-the-art algorithms.

\vspace{0.5cm}
\end{abstract}

\clearpage
\tableofcontents
\clearpage

\section{Introduction}

Non-negative least squares~(NNLS) problem that aims to minimize the sum of squares of the errors is one of the most fundamental problems in numeric analysis. It has been widely used in scientific computation and data mining to approximate observations~\cite{Chen2009Non}. Particularly, in many fields such as image processing, computer vision, text mining, environmetrics, chemometrics, or speech recognition, we often need to estimate a large number of observations $b \in \mathcal{R}^{d}$ by a set of measurements or basis factors $\{A_i\}$ contained in a matrix $A \in \mathcal{R}^{d \times n}$, where the popular task is to approximate $b$ by minimizing $\frac{1}{2}||Ax - b||^2_2$. Moreover, in comparison to least squares (LS), NNLS has more concisely  interpretable solutions,
of which non-negative coefficients $x \in \mathcal{R}_+^n$ can be interpreted as contributions of the measurements over the observations. In the contrast, mixed-sign coefficients of LS solutions are uninterpretable because they lead to overlapping and mutual elimination of the measurements.

Since 1990s, the methods of nonnegative matrix or tensor factorizations have widely used NNLS to achieve a low-rank representation of nonnegative data~\cite{Kolda2009Tensor,Zhang2011a}. Particularly, the low-rank representation transfers data instances into a lower-dimensional space of components or sources to obtain higher speed and accuracy, and more concise interpret-ability of data processing that are essential in applications of signal and image processing, machine learning, and data mining~\cite{Chen2009Non}. 

Gradient methods are usually employed  to solve NNLS because there is not any general formula of solutions, although NNLS is a convex optimization problem having a unique solution. As other gradient methods, the performance of NNLS algorithms mainly depends on selecting appropriate directions to optimize the objective function. To do this, most effective algorithms remove redundant variables based on the concept of active sets~\cite{Bro1997,Chen2009Non} in each iteration, although they are different in the employed strategies~\cite{Chen2009Non}. Fundamentally, these  algorithms are based on the observation that several variables can be ignored if they are negative when the problem is unconstrained~\cite{Bro1997,Lawson1974,Van2004fast}. In other words, NNLS can be considered as an unconstrained problem in a subspace of several variables~\cite{Kim2013} that are positive in the optimal solution. In addition, algorithms using the second derivative~\cite{Bro1997,Lawson1974,Van2004fast} discover effective directions that more effectively reduce the objective function. However, these approaches have two main drawbacks: invertibility of $A^TA$ and its heavy computation, especially for the methods recomputing $(A^TA)^{-1}$ several times for different passive sets. Hence, other algorithms using only the first derivative~\cite{Franc2005seq,Kim2013,Potluru2012frugal} can be more effective for large-scale least squares problems. 

In our view, to discover more appropriate gradient directions is critical to enhance the performance of NNLS algorithms. In this paper, we propose a fast robust iterative algorithm to solve that issue called anti-lopsided algorithm. The main idea of this algorithm is to transfer the original problem into an equivalent problem, by which scaling problems in the first order methods are significantly reduced to obtain more effective gradient directions. The proposed algorithm has significant advantages:

\begin{itemize}
	\item \textbf{Fastness}: The proposed algorithm attains linear convergence rate of $(1-\frac{1}{2||Q||_2})^k$ in the subspace of passive variables, where $\sqrt{n} \leq ||Q||_2 \leq n$ and $n$ is the dimension size of solutions. In addition, it does not require computing the inverse of matrices $(A^TA)^{-1}$ because it is totally based on the first derivative.
	
	\item \textbf{Robustness}: It can stably work in ill-conditioned cases since it is totally based on the first derivative,
	
	\item \textbf{Effectiveness}: The experimental results are highly comparable with the state-the-art methods,
	
	\item \textbf{Convenience}: Interestingly, the method is convenient to be implemented, paralleled and distributed since it is totally based on the first order derivative and gradient methods using the exact line search.
\end{itemize}

The rest of paper is organized as follows. Section~\ref{sec:background} discusses the background and related work of least square problems. Section~\ref{sec:antilopsided} mentions the details of our proposed algorithm anti-lopsided. Subsequently, the theoretical analysis is discussed in Section~\ref{sec:antilopsided}. Finally, Section~\ref{sec:evaluation} shows the experimental results, and Section~\ref{sec:conclusion} summarizes the main contributions of this paper.

\section{Background and Related Work} \label{sec:background}

In this section, we introduce the non-negative least square (NNLS) problem, its equivalent problem as well as non-negative quadratic problem (NQP), and significant milestones of the algorithmic development for NNLS. 

\subsection{Background}

Non-negative least square (NNLS) can be considered one of the most central problems in data modeling to estimate the parameters of models for describing the data~\cite{Chen2009Non}. It comes from scientific applications where we need to estimate a large number of vector observations $b \in \mathcal{R}^{d}$ by a set of measures or basis factors $\{A_i\}$ contained in a matrix $A \in \mathcal{R}^{d \times n}$. The common task is to approximate an observation $b$ by the measures by minimizing $\frac{1}{2}||Ax - b||^2_2$. Hence, we  can define NNLS as follows:

\begin{definition}
	\it Given $n$ measurement vectors $A = [A_1, A_2,$ $..., A_n] \in \mathcal{R}^{d \times n}$ and an observed vector $b \in \mathcal{R}^d$, find an optimal solution $x$ of the optimization problem:
	\begin{equation}\label{eq:NNL}
		\begin{aligned}
			& \underset{x}{\text{minimize}}
			&& \frac{1}{2}||Ax - b||^2_2 \\
			&\text{subject to}
			&& x \succeq 0\\
			&\text{where}
			&& A \in \mathcal{R}^{d \times n}, b \in \mathcal{R}^d
		\end{aligned}
	\end{equation}
\end{definition}

Obviously, it can be equivalently turned  into a non-negative quadratic programming (NQP) problem:

\begin{equation}\label{eq:NQP}
	\begin{aligned}
		& \underset{x}{\text{\textit{minimize}}}
		& & f(x) = \frac{1}{2}x^THx + h^Tx \\
		& \text{\textit{subject to}}
		& & x \succeq 0.\\
		& \text{\textit{where}}
		& & H = A^TA,  h = -A^Tb\\
	\end{aligned}
\end{equation}

From this NQP formulation, these problems are convex since $Q$ is positive semidefinite and the non-negativity constraints form a convex feasible set. In this paper, we will solve Problem \ref{eq:NQP} instead of Problem \ref{eq:NNL}.

\subsection{Related Work} In last several decades of development,  many different approaches have been proposed to tackle the NNLS problem, which can be divided into two main groups: active-set methods and iterative methods~\cite{Chen2009Non}.

Active-set methods are based on the observation that variables can be divided into subsets: active and passive variables~\cite{Gill1987}. Particularly, the active set contains variables being zero or negative when solving the least square problem without concerning non-negative constraints, otherwise the other variables belongs to the passive set. 
The active-set algorithms employ the fact that if  the active set is identified, the values of the passive variables in NNLS are their values in the unconstrained least squares solution when removing active variables that will be set to zero. Unfortunately, these sets are unknown in advance. Hence, a number of iterations is employed to find out the passive set, each of which needs to solve a unconstrained least squares problem on the passive set to update the passive set and the active set. 
Concerning the significant milestones of the active-set methods, Lawson and Hanson (1974)~\cite{Lawson1974} proposed a standard algorithm for active-set methods. Subsequently, Bro and De Jong (1997)~\cite{Bro1997} avoided unnecessary re-computations on multiple right-hand sides to speed up the basic algorithm~\cite{Lawson1974}. 
Finally,  Dax (1991)~\cite{Dax1991} proposed selecting a good starting point by \textit{Gauss-Seidel} iterations and moving away from a \quotes{dead point} to reduce the number of iterations.

On the other hand, the iterative methods use the first order gradient on the active set to handle multiple active constraints in each iteration, while the active-set methods only handle one active constraint~\cite{Chen2009Non}.  Hence, the iterative methods can deal with larger-scale problems~\cite{Kim2006NNLS,Kim2013} than the active-set methods. However, they are still not guaranteed the convergence rate. More recently, accelerated  methods~\cite{Nesterov1983} and proximal  methods~\cite{Parikh2013proximal} having a fast convergence $O(1/k^2)$~\cite{Guan2012} only require the first order derivative. However, one major disadvantage of the accelerated methods is that they require a big number of iterations when the solution requires high accuracy because the step size is limited by $\frac{1}{M}$ that is usually small for large-scale NNLS problems with big matrices; where $M$ is Lipschitz constant.

In summary, active-set methods and iterative methods are two major approaches in solving NNLS. The active-set methods accurately solve nonnegative least square problem, but require a huge computation for solving unconstrained least squares problems and are unstable when $A^TA$ is ill-conditioned.  The iterative methods have more potential in solving large-scale NNLS because they can handle multiple active constraints per each iteration. In our view, the iterative methods are still ineffective due to the scaling problem on variables that seriously affects the finding of appropriate gradient directions. Therefore, we propose the anti-lopsided algorithm as an iterative method that re-scales variables to reduce the scaling problem, obtain appropriate gradient directions, and achieve linear convergence on the sub-space of passive variables.

\section{Anti-lopsided algorithm}\label{sec:antilopsided}

In this section, we discuss our new idea to develop a fast and robust iterative algorithm for large-scale NNLS problem. For more readability, we separate the proposed algorithm into two module algorithms: Algorithm~\ref{algo:NNLS} for solving NNLS problem by first transforming it into a NQP problem and re-scaling variables by which the scaling problem of variables is significantly reduced, then calling Algorithm \ref{algo:NQPByExactLineSearch} for solving NQP problem by a first-order gradient method using the exact line search.

\begin{figure}
	\centering{
		\hspace*{-20pt}
		\includegraphics[scale=0.5]{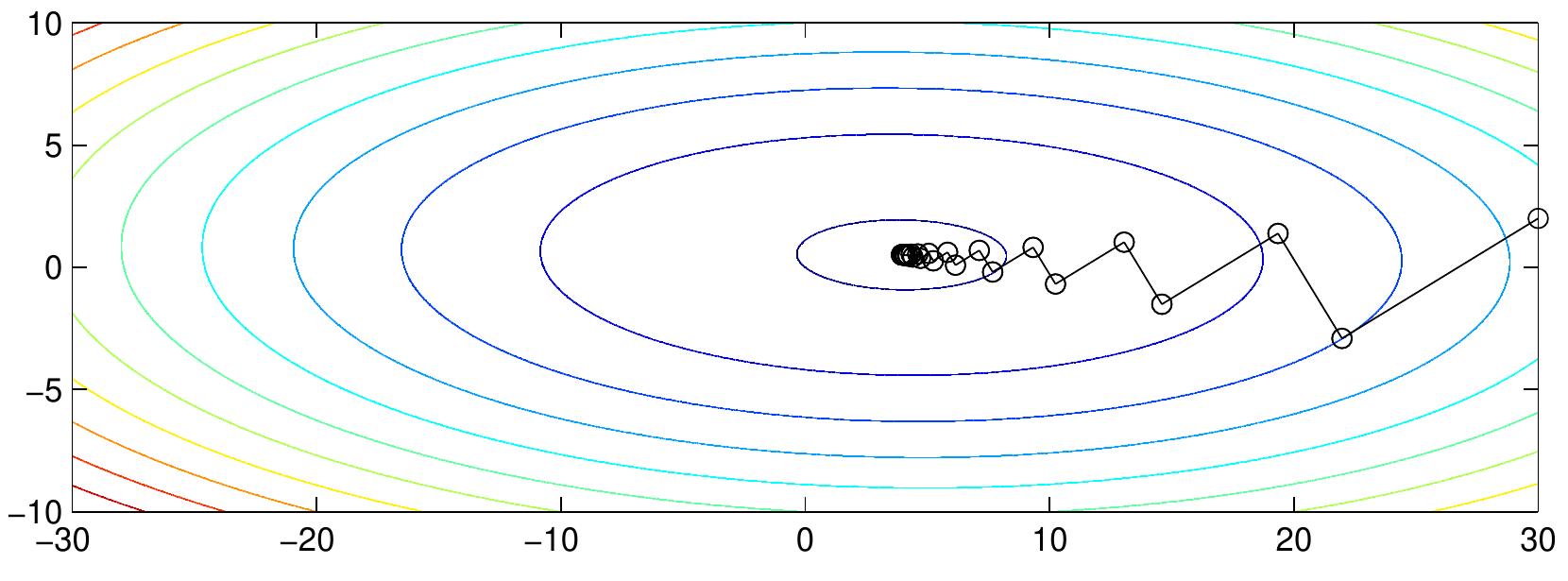}
		\caption{40 optimizing steps in the iterative exact line search method using  the first order derivative for the function~\ref{eq:scaling} starting at $x_0=[30\ 2]^T$}
		\label{fig:not_using_scissored}}
	\vspace*{-10pt}
\end{figure}

\begin{algorithm2e}  
	\caption{Anti-lopsided algorithms for NNLS}\label{algo:NNLS} 
	\KwIn{$A \in \mathcal{R}^{d\times n}$; $b \in \mathcal{R}^d$}
	\KwOut{$x$ minimizing $\frac{1}{2}||Ax - b||^2_2$\\
		\ \ \ \ \ \ \ \ \ \ \ \ \ subject to: $x \succeq 0$
	}
	\Begin{
		$H = A^TA$\;
		$Q = \frac{H}{\sqrt{diag(H)diag(H)^T}}$\;
		$q = \frac{-A^Tb}{\sqrt{diag(H)}}$\;
		$y$ = solveNQP$(Q, q)$/*by Algorithm \ref{algo:NQPByExactLineSearch}*/\;
		$x = \frac{y}{\sqrt{diag(H)}}$\;
		\Return{x}
	}
\end{algorithm2e}

It is well known from the literature that iterative methods using  the first order derivative can be suitable for designing large-scale algorithms for NNLS. However, this approach heavily depends on the scaling of variables~\cite{Boyd2004convex}. For example, if we employ the iterative exact line search method using  the first order derivative to optimize the function~\ref{eq:scaling}, we need 40 iterations to reach the optimal solution (see Figure~\ref{fig:not_using_scissored}):

\begin{equation}\label{eq:scaling}
f(x) = \frac{1}{2}x^T\begin{bmatrix} 1&0.1\\ 0.1&9 \end{bmatrix}x + [-4 -5]x
\end{equation}

Hence, we re-scale variables into a new space, in which new variables have more balanced roles and for that reason we call our algorithm as anti-lopsided algorithm. Particularly, we re-scale:

\begin{equation}\label{eq:Rescaling}
\begin{aligned}
&\ && x = \frac{y}{\sqrt{diag(H)}}  &\text{ or }&& x_i = \frac{y_i}{\sqrt{H_{ii}}}\ \ \forall i\\
\end{aligned}
\end{equation}

It is noticeable that $H_{ii}=A_i^TA_i=||A_i||_2^2 \geq 0$. For the special case $H_{ii}=0$, we give $x_i=y_i$. After re-scaling variables, the original Problem~\ref{eq:NQP}  is equivalently transformed into NQP Problem~\ref{eq:NQPQ}:
\begin{equation}\label{eq:NQPQ}
\begin{aligned}
&\underset{y}{\text{\textit{minimize}}} &&f(y) = \frac{1}{2}y^TQy + q^Ty \\
&\text{\textit{subject to}} &&y \succeq 0\\
&\text{\textit{where}} && Q_{ij}=\frac{H_{ij}}{\sqrt{H_{ii}H_{jj}}}; q_i = \frac{h_i}{\sqrt{H_{ii}}}\\
\end{aligned}
\end{equation}

\begin{remark}
	Consider the values of matrix $Q$, we have: 
	\begin{itemize}
		\item $Q_{ii}=\frac{H_{ii}}{\sqrt{H_{ii}^2}}=\cos(A_i, A_i)=1,\ \forall\ i=1,...,n$
		\item $Q_{ij}=\frac{H_{ij}}{\sqrt{H_{ii}H_{jj}}}=\cos(A_i,A_j),\ \forall\ 1 \leq i \neq j \leq n$
	\end{itemize}
\end{remark}

Obviously, the scaling problem of variables is significantly reduced because the roles of variables in the function become more balanced. The shape of function is transformed from ellipse-shaped(see Figure~\ref{fig:not_using_scissored}) to ball-shaped (see Figure~\ref{fig:using_antilopsided_scissored}). Hence, the first order methods can much more effectively work. For example, we need only 3 iterations instead of 40 iterations to reach the optimal solution for the function~\ref{eq:scaling} with the same initial point $y_0$, where ${y_0}_i={x_0}_i.H_{ii},\ \forall i$ (see Figure~\ref{fig:using_antilopsided_scissored}).

Subsequently, we discuss on Algorithm \ref{algo:NQPByExactLineSearch} for solving NQP with input $Q$ and $q$. For each iteration, the objective function is optimized on the passive set:
\begin{equation*}
P(x) = \{x_i | x_i > 0 \text{ or } \nabla f_i(x) < 0\}
\end{equation*}

Hence, the first order gradient will be projected into the sub-space of the passive set $\nabla \bar{f} = [\nabla f]_{P(x)}$ ($\nabla \bar{f}_i = \nabla f_i \text{ for } i \in P(x),\text{ otherwise } \nabla \bar{f}_i = 0$). Noticeably, the passive set can change through iterations, and Algorithm \ref{algo:NQPByExactLineSearch} is converged when $P(x) = \emptyset$ or $||\nabla \bar{f}||^2 < \epsilon$.  In addition, the orthogonal projection on the sub-space of passive variables $x = [x_k + \alpha \nabla \bar{f}]_+$ is trivial~\cite{Kim2013} since NQP problem~\ref{eq:NQPQ} is a strongly convex problem on a convex set.

\begin{algorithm2e} [h]
	\caption{Gradient Method using Exact Line Search for NQP}
	\label{algo:NQPByExactLineSearch} 
	\KwIn{$Q \in \mathcal{R}^{n \times n}$; $q \in \mathcal{R}^n$}
	\KwOut{$x$ minimizing $f(x) = \frac{1}{2}x^TQx + q^Tx$\\
		\ \ \ \ \ \ \ \ \ \ \ \ \ subject to: $x \succeq 0$}
	\Begin{
		$x^0 = \text{\textbf{0}}$\;
		$\nabla f = q$\;
		\Repeat{$P(x) = \emptyset$ or $||\nabla \bar{f}||^2 < \epsilon$}{
			/*Remove active variables and keep passive variables*/\;
			$\nabla \bar{f} = \nabla f[x > 0 \text{ or } \nabla f < 0]$\;
			$\alpha = \argmin{\alpha} f(x_k - \alpha \nabla \bar{f}) = \frac{||\nabla \bar{f}||^2_2}{\nabla \bar{f}^TQ\nabla \bar{f}}$\;
			$x_{k+1} = [x_k - \alpha \nabla \bar{f}]_+$\;
			$\nabla f = \nabla f + Q(x_{k+1} - x_k)$\;
		}
		\Return{$x_{k}$}
	}
\end{algorithm2e}

Regarding computing $a$ of Algorithm \ref{algo:NQPByExactLineSearch}, we have:

\begin{equation*}
\begin{aligned}
& f(x - \alpha \nabla \bar{f}) = - \alpha \nabla \bar{f}^T[Qx+q] + \frac{\alpha^2}{2}\nabla \bar{f}^TQ\nabla \bar{f}  + \text{C}\\
\Rightarrow & \frac{\partial f}{\partial \alpha} = 0 \Leftrightarrow \alpha = \frac{\nabla \bar{f}^T (Qx+q)}{\nabla \bar{f}^TQ\nabla \bar{f}} = \frac{||\nabla \bar{f}||^2_2}{\nabla \bar{f}^TQ\nabla \bar{f}}\\
&\text{ where $C$ is constant }
\end{aligned}
\end{equation*}

\begin{figure}
	\centering{
		\hspace*{-20pt}
		\includegraphics[scale=0.5]{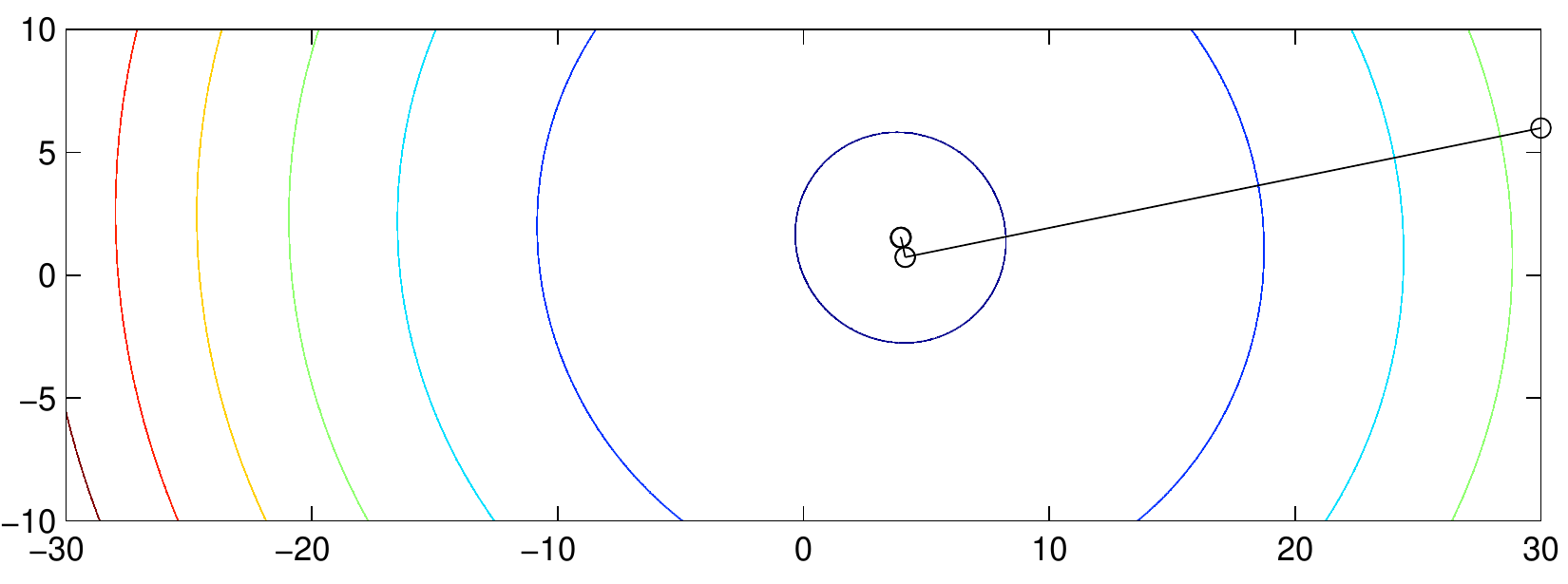}
		\caption{3 optimizing steps in the iterative exact line search method using  the first order derivative after applying anti-lopsided steps starting at $y_0 = x_0.*\sqrt{diag(H)}$}
		\label{fig:using_antilopsided_scissored}}
	\vspace*{-10pt}
\end{figure}

\section{Theoretical Analysis}\label{sec:analysis}

In this section, we investigate the convergence and complexity of the proposed algorithm.

\subsection{Convergence}

Concerning the convergence rate, our method argues Barzilai and Borwein's note that NNLS is an unconstrained optimization on the passive set of variables~\cite{Bro1997}. Moreover, the orthogonal projection on the sub-space of passive variables $x = [x_k + \alpha \nabla \bar{f}]_+$ is trivial~\cite{Kim2013} since NNLS and its equivalent problem (NQP) is a strongly convex problem on a convex set. Hence, in this section, we consider the convergence rate of Algorithm \ref{algo:NQPByExactLineSearch} in the sub-space of passive variables.

Let $f(x) = \frac{1}{2}x^TQx + q^Tx$, where  $Q_{ij} = \cos(A_i, A_j) = cos(a_i, a_j)$ and where $a_k = \frac{A_k}{||A_k||_2}$ is the unit vector $A_k$, $\forall k$. 

Since $f(x)$ is strongly convex, we have:
\begin{itemize}
	\item $\exists m, M > 0$ satisfy $ mI \preceq \nabla^2 f \preceq MI$
	\item $\forall x, y: f(y) \geq f(x) + \langle \nabla f(x), (y - x)\rangle + \frac{m}{2} ||y - x||^2 $
	\item $\forall x, y: f(y) \leq f(x) + \langle \nabla f(x), (y - x)\rangle + \frac{M}{2} ||y - x||^2 $
\end{itemize}

%
%
Based on~\cite{Caramanis_ee381v:}, we have:
\begin{theorem}
	\label{theo:convergence}
	After $(k+1)$ iterations, $f(x^{k+1}) - f^* \leq (1 - \frac{m}{M})^k (f(x^0) - f^*)$, where $f^*$ is the minimum value of $f(x)$.
\end{theorem}
\begin{proof}
	Since $f(y) \leq f(x) + \langle\nabla f, y-x\rangle + \frac{M}{2}||y - x||^2_2\ \ \forall x, y$ selecting $y=x - \frac{1}{M}\nabla f$ and $x^+$ is the updated value of $x$ after an iteration by the first order gradient using exact line search, we have:
	\begin{equation}
	\begin{aligned}
	f(x^+) \leq & f(x - \frac{1}{M}\nabla f)
	\leq f(x) - \frac{1}{M}||\nabla f||^2_2 + \frac{M}{2}(\frac{1}{M})^2||\nabla f||^2_2\\
	\leq & f(x) - \frac{1}{2M}||\nabla f||^2_2
	\end{aligned}
	\end{equation}
	Hence, for the minimum value $f^*$ of the objective function, we have:
	\begin{equation}\label{eq:cm_ht_1}
	f(x_{k + 1}) - f^* \leq (f(x_k) - f^*) - \frac{1}{2M} ||\nabla f||^2_2
	\end{equation}
	
	Consider $f(y) = f(x) + \langle\nabla f, y-x\rangle + \frac{m}{2}||y - x||^2_2$ (fixing $x$) is a convex quadratic function of $y$. Hence, $f(y)$ minimizes when $\nabla f(y) = 0 \Leftrightarrow y = \tilde{y}=x-\frac{1}{m}\nabla f$. 
	In addition, since $f(y) \geq f(x) + \langle\nabla f, y-x\rangle + \frac{m}{2}||y - x||^2_2\ \ \forall x, y$, we have:
	\begin{equation}
	\begin{aligned}
	f(y) \geq & f(x) + \langle \nabla f, y - x\rangle + \frac{m}{2}||y - x||
	\geq f(x) + \langle \nabla f, \tilde{y} - x\rangle + \frac{m}{2}||\tilde{y} - x||\\
	= & f(x) - \frac{1}{2m}||\nabla f||^2_2 \ \ \forall x, y\\
	\end{aligned}
	\end{equation}
	Selecting $y = x^*$ and $x = x_{k}$ where $x^*$ is the optimal solution, we have:
	\begin{equation}\label{eq:cm_ht_2}
	-||\nabla f||^2_2 \leq 2m (f^*- f(x_k))
	\end{equation}
	From (\ref{eq:cm_ht_1}) and (\ref{eq:cm_ht_2}), we have the necessary result:
	\begin{equation*}
	f_{k+1} - f^* \leq (1-\frac{m}{M}) (f(x_k) - f^*) \leq (1-\frac{m}{M})^k (f(x_0) - f^*)
	\end{equation*}
\end{proof}

\begin{lemma}
	\label{lemma:Mm}
	For $f(x) = \frac{1}{2}x^TQx + q^Tx$, then $\frac{1}{2}I \preceq \nabla^2 f \preceq ||Q||_2I$ and $\sqrt{n} \leq ||Q||_2 \leq n$, where  $Q_{ij} = cos(a_i, a_j)$, $a_i, a_j$ are unit vectors,  and $||Q||_2 = \sqrt{\sum_{i=1}^{n}\sum_{j=1}^{n}Q_{ij}^2}$.
\end{lemma}

\begin{proof}
	We have $\nabla^2 f = Q$, and
	
	$\frac{1}{2}x^TIx \leq \frac{1}{2}(\sum_{i=1}^{n}x^2_i) + \frac{1}{2}||\sum_{i=1}^{n}x_i a_i||^2 = \sum_{i=1}^{n}\sum_{j=1}^{n}Q_{ij}x_ix_j=x^TQx$ for $\forall x$
	$\Rightarrow \frac{1}{2}I \preceq \nabla^2 f$.
	
	Moreover, based on Cauchy-Schwarz inequality, we have:
	\begin{equation*}
	\begin{aligned}
	\ &(\sum_{i=1}^{n}\sum_{j=1}^{n}Q_{ij}x_ix_j)^2 \leq (\sum_{i=1}^{n}\sum_{j=1}^{n} Q_{ij}^2) (\sum_{i=1}^{n}\sum_{j=1}^{n} ({x_i}{x_j})^2) \\
	\Rightarrow &\sum_{i=1}^{n}\sum_{j=1}^{n}Q_{ij}x_ix_j \leq \sqrt{||Q||^2_2 (\sum_{i=1}^{n} {x_i}^2)^2}\\
	\Leftrightarrow & x^TQx \leq ||Q||_2x^TIx \ \ (\forall x) \ \ \Leftrightarrow Q \preceq ||Q||_2I\\
	\end{aligned}
	\end{equation*}
	
	Finally, $\sqrt{n} = \sqrt{\sum_{i=1}^{n} Q_{ii}^2} \leq ||Q||_2 = \sqrt{\sum_{i=1}^{n}\sum_{j=1}^{n}Q_{ij}^2} \leq \sqrt{n^2} = n$ since $-1 \leq Q_{ij} = \cos(a_i, a_j) \leq 1$.
\end{proof}

From Theorem~\ref{theo:convergence} and Lemma~\ref{lemma:Mm}, setting $m=\frac{1}{2}$ and $M=||Q||_2$, we have:
\begin{theorem}
	After $k+1$ iterations, $f(x^{k+1}) - f(x^*) \leq (1 - \frac{1}{2||Q||_2})^k (f(x^0) - f(x^*))$, where $\sqrt{n} \leq ||Q||_2 \leq n$ and $n$ is the dimension of $x$.
\end{theorem}

\subsection{Complexity}

Concerning the average complexity of Algorithm \ref{algo:NQPByExactLineSearch}, we consider the important operators:
\begin{itemize}
	\item Line 7: The complexity of computing $||\nabla \bar{f}||^2_2$ is $\mathcal{O}(n)$, and that of $\nabla \bar{f} Q \nabla \bar{f}$ is $\mathcal{O}(nS(n))$.
	\item Line 9: The complexity of computing $Q(x_{k+1} - x_k)$ is $\mathcal{O}(nS(n))$.
\end{itemize}
\noindent where $S(n)$ is the average number of non-zero elements of $\nabla \bar{f}$. Noticeably, the sparsity of $(x_{k+1} - x_k)$ equals to the sparsity of $\nabla \bar{f}$ since $x_{k+1} = x_k + \alpha \nabla \bar{f}$.

\begin{lemma} \label{lemma:complexity}
	The average complexity of Algorithm \ref{algo:NQPByExactLineSearch} is $\mathcal{O}(\bar{k}nS(n))$, where $\bar{k}$ is the average number of iterations.
\end{lemma}

Therefore, based on Lemma~\ref{lemma:complexity}, if we consider computing $A^TA$ and $A^Tb$ in $\mathcal{O}(dn+dn^2)$, we have:

\begin{theorem}
	The average complexity of Algorithm \ref{algo:NNLS} is $\mathcal{O}(dn+dn^2+\bar{k}nS(n))$, where $\bar{k}$ is the average number of iterations and $S(n)$ is the average number of non-zero elements of $\nabla \bar{f}$.
\end{theorem}

%
%

\section{Experimental Evaluation}\label{sec:evaluation}

In this section, we investigate the convergence speed, running time  and optimality of the proposed algorithm in comparison to state-of-the-art algorithms belongs different research directions: active-set methods, iterative methods and accelerated methods. Particularly, we compare our algorithm \textbf{Antilop} with the following algorithms:
\begin{itemize}
	\item \textbf{Fast}: This is a modified effective version of active-set methods according to Bro R., de Jong S., Journal of Chemometrics, 1997~\cite{Bro1997}, which is developed by S. Gunn~\footnote{http://web.mit.edu/~mkgray/matlab/svm/fnnls.m}. This algorithm can be considered as one of the fastest algorithms of active-set methods.
	
	\item \textbf{Nm}: This is a non-monotonic fast method for large-scale nonnegative least squares based on iterative methods~\cite{Kim2013}. The source code is downloaded from~\footnote{http://suvrit.de/work/progs/nnls.html}.
	
	\item \textbf{Accer}: This is a Nesterov accelerated method with convergence rate $O(1/k^2)$~\cite{Guan2012}. The source code is extracted from a module in the paper~\cite{Guan2012}. The source code is downloaded from~\footnote{https://sites.google.com/site/nmfsolvers/}.
	
	\item \textbf{Anti+Acc}: This is a Nesterov accelerated method adding the anti-lopsided steps (from Line 2 to Line 4 in Algorithm~\ref{algo:NNLS}). This evaluation investigates the effectiveness of the anti-lopsided steps for algorithms that only use the first derivative.
\end{itemize}

To be a fair comparison, we evaluate the five algorithms on various test-cases which are randomly generated. Particularly, we use a set of test-cases with $n = 4000$ and $d=1.5n=6000$ since usually $d > n$. Furthermore, we design 6 test-cases \{T1, ..., T6\}, which have high potential to happen in practice. To generate test-cases having nontrivial solutions, matrix $A$ and a vector $x^*$ are randomly generated, and $b=Ax^*$. The sign of values in $A$ and $x^*$ can be non-negative($+$) or mixed-sign($\pm$); and measure vectors $\{A_i\}_{i=1}^d$ have the same (SAM), randomized (RAN), or various (VAR) lengths, which are enumerated in Table~\ref{table:testcases}. Test-cases containing $A, x^*$ negative are ignored because they will have $x=0$ as the optimal solution. Concerning the optimal value $f^*$ to evaluate, hence, it is $0$ for nonnegative test cases; and it is the best minimum value found by all the methods  for mixed-sign test cases because we do not know it for the test-cases in advance.

For each test-case, we randomize 5 sub-tests having different sparsity of vectors $A_k$ and $x^*$ from 0\% to 40\% to create the objective functions having different values of Lipschitz constant and various cases of solution (see Table~\ref{table:varity}).

\begin{table}
	\caption{Summery of the kinds of test-cases}
	\label{table:testcases}
	\centering
	{\begin{tabular}{ccccccc} 
			\hline\noalign{\smallskip}
			\ & \textbf{T1} & \textbf{T2} & \textbf{T3} & \textbf{T4} & \textbf{T5} & \textbf{T6} \\
			\noalign{\smallskip}\hline\noalign{\smallskip}
			$A, x^*$ & $+$ & $\pm$ & $+$ & $\pm$ & $+$ & $\pm$ \\ 
			Length of $A_k$ & SAM & RAN & VAR & SAM & RAN & VAR \\
			\noalign{\smallskip}\hline
		\end{tabular}}
	\end{table}
	
	\begin{table}
		\caption{Range of $||A||^2_2$ and $||A^TA||^2_2$ in test cases}
		\label{table:varity}
		\centering
		{\begin{tabular}{lcc} 
				\hline\noalign{\smallskip}
				\ & $||A||_2^2$ & $||A^TA||_2^2$ \\
				\noalign{\smallskip}\hline\noalign{\smallskip}
				T1 & $1.8e$+$03$\ -\ $4.7e$+$06$ & $1.1e$+$05$\ -\ $1.4e$+$13$\\
				T2 & $1.7e$+$03$\ -\ $3.5e$+$06$ & $9.4e$+$04$\ -\ $3.9e$+$11$\\ 
				T3 & $2.4e$+$03$\ -\ $3.1e$+$06$ & $5.6e$+$05$\ -\ $6.0e$+$12$\\ 
				T4 & $2.3e$+$03$\ -\ $3.9e$+$06$ & $1.6e$+$05$\ -\ $7.5e$+$12$\\ 
				T5 & $2.6e$+$03$\ -\ $2.9e$+$06$ & $2.1e$+$05$\ -\ $5.1e$+$12$\\ 
				T6 & $9.9e$+$02$\ -\ $5.0e$+$06$ & $3.1e$+$04$\ -\ $1.5e$+$13$\\ 
				\noalign{\smallskip}\hline
			\end{tabular}}
			
		\end{table}
		
		Since the algorithms can run for a long time, we must stop them when no finding better solution with a smaller value of $||\nabla \bar{f}||^2$, reaching the maximum number of iterations $k = 5n = 2.10^4$, or the maximum of running time is 1600(s) that is the running time of the traditional accurate method Fast~\cite{Bro1997} for all the test-cases.
		
		In addition, we develop Algorithm \ref{algo:NNLS} and Algorithm~\ref{algo:NQPByExactLineSearch} in Matlab to easily compare them with other algorithms. Furthermore, we set system parameters to use only 1 CPU for Matlab and the IO time is excluded. All source codes of our algorithm and generating test-cases, and 30 used test-cases are published ~\footnote{https://bitbucket.org/[will-publish]/antilopsidednnls}.
		
		\subsection{Convergence}
		\begin{figure*}
			\centering{
				\hspace*{-20pt}
				\includegraphics[scale=0.6]{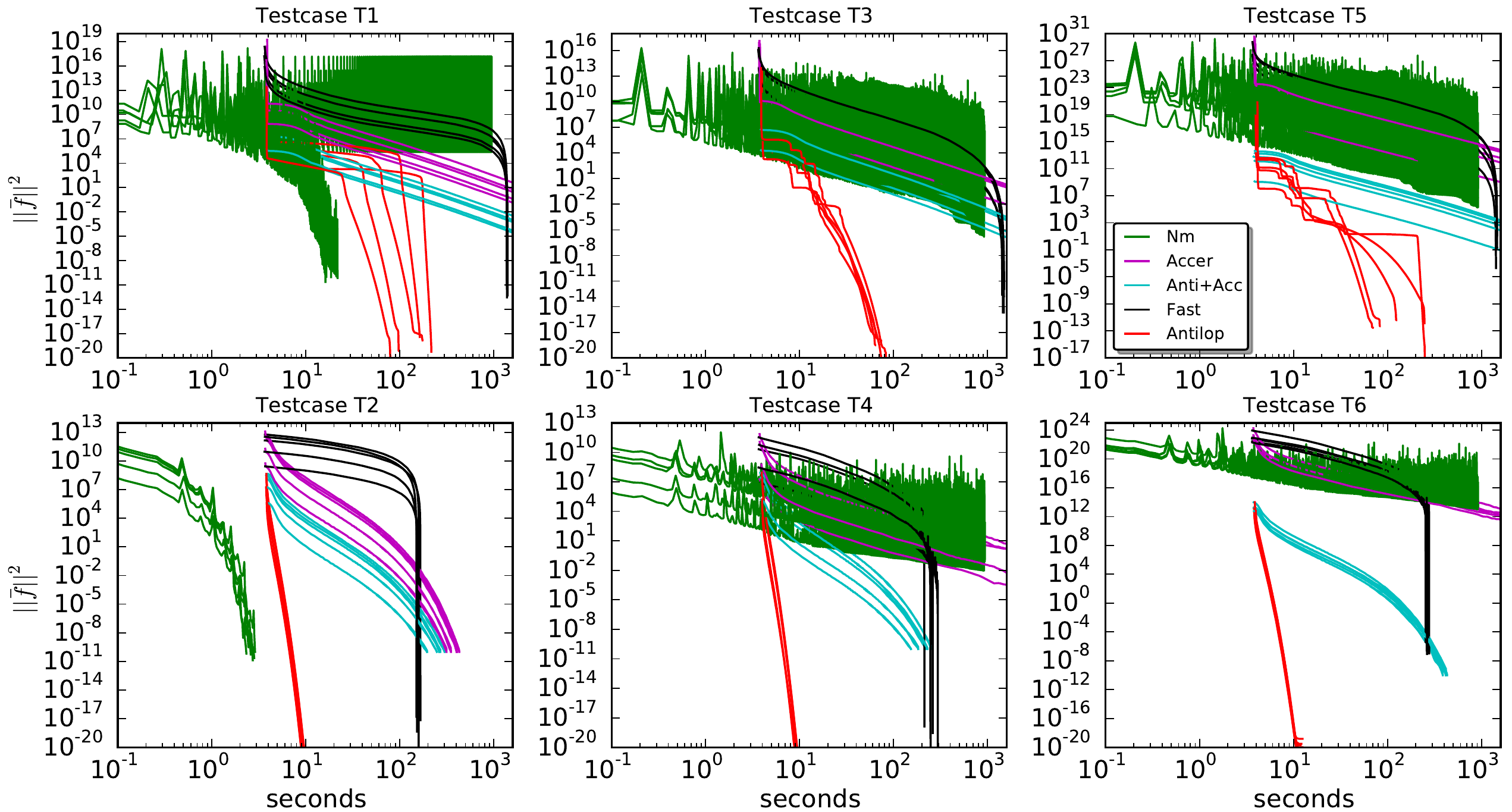}
				\caption{$||\bar{f}||^2_2$ during running time}
				\label{fig:convergence_deri}}
			\vspace*{-10pt}
		\end{figure*}
		
		In this section, we investigate the convergence speed of the square of derivatives $||\bar{f}||^2_2 \rightarrow 0$ (see Figure~\ref{fig:convergence_deri}) and the difference between the values of objective function and the optimal values $(f(x_k) - f^*+1) \rightarrow 10^0$ during the running time (see Figure~\ref{fig:convergence_obj}). The results clearly show that our proposed algorithm and  algorithm Fast~\cite{Bro1997} converge to the optimal values in all 30 test-cases. Remarkably, our algorithm converges to the optimal values much faster than the other methods. Moreover, interestingly, algorithm Anti+Acc works more effectively than algorithm Accer. This proves that the anti-lopsided transformation makes the convergence of iterative methods using the first derivative more faster because the transformation significantly reduced the scaling problem of variables to obtain more appropriate gradient directions.
		
		\begin{figure*}
			\centering{
				\hspace*{-20pt}
				\includegraphics[scale=0.6]{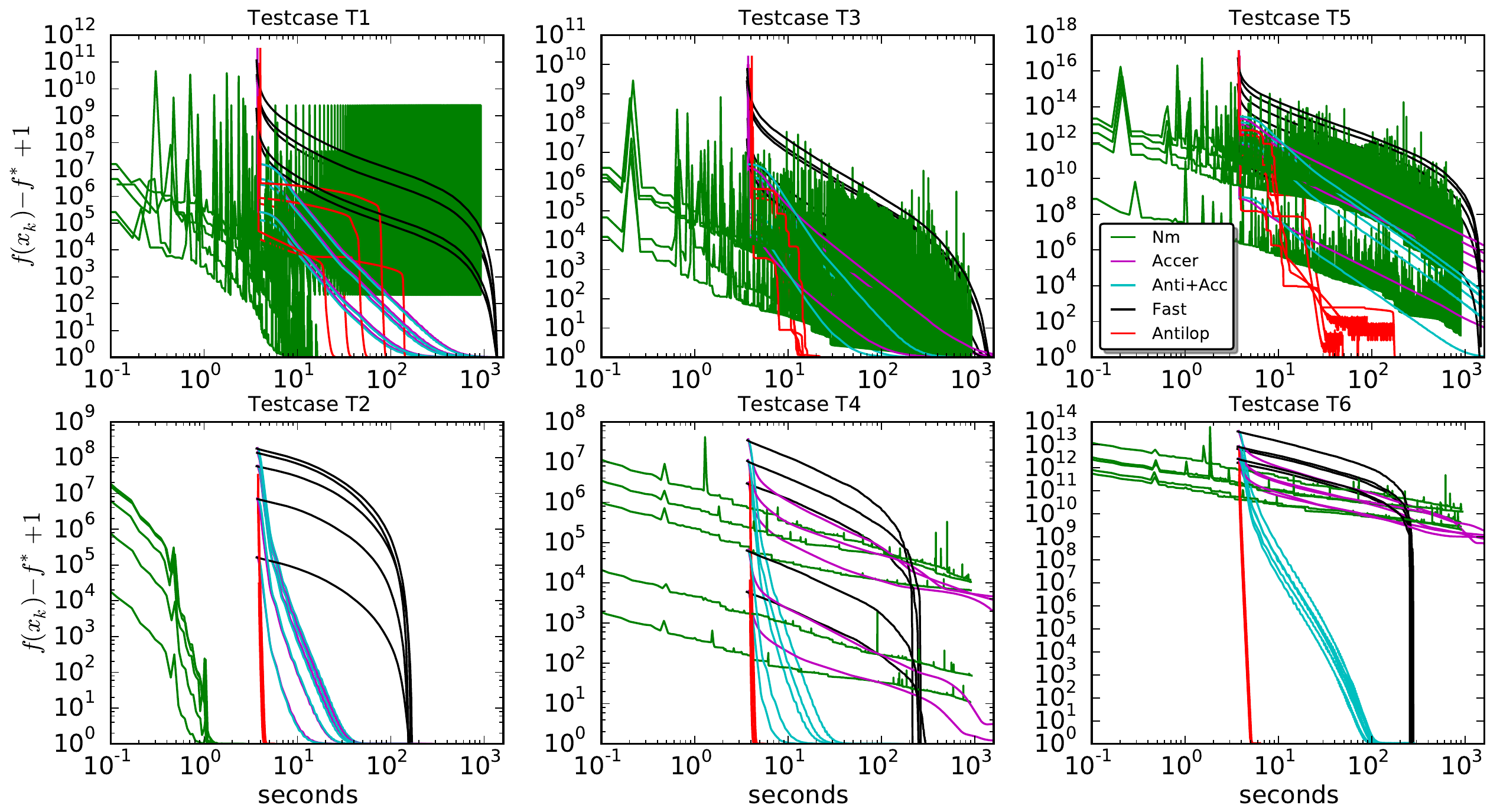}
				\caption{$(f(x_k) - f^* + 1)$ during running time}
				\label{fig:convergence_obj}}
			\vspace*{-10pt}
		\end{figure*}
		
		\subsection{Running Time}
		
		\begin{table}
			\caption{Average of running time in seconds for test-cases $T_1,...,T_6$}
			\label{table:running_time}       
			\centering
			\begin{tabular}{lccccc}
				\hline\noalign{\smallskip}
				Test-case  & \textbf{Antilop} &  \textbf{Fast} &  \textbf{Nm} &  \textbf{Accer} &\textbf{Anti+Acc}\\ 
				\noalign{\smallskip}\hline\noalign{\smallskip}
				T1 &  \textbf{146.8} &  1384.2 &  205.2 &  1600.0 &  1600.0 \\
				T2 &  11.9 &  160.6 &  \textbf{2.7} &  253.4 &  166.7 \\
				T3 &  \textbf{71.5} &  1471.9 &  917.6 &  1600.0 &  1600.0 \\ 
				T4 &  \textbf{12.5} &  255.6 &  906.4 &  1600.0 &  138.5 \\ 
				T5 &  \textbf{154.7} &  1455.3 &  905.3 &  1600.0 &  1600.0 \\ 
				T6 &  \textbf{12.3} &  265.7 &  906.9 &  1600.0 &  276.2 \\ 
				\noalign{\smallskip}\hline
			\end{tabular}
		\end{table}
		
		In the final comparison, we compare the algorithms in the average of running time on 5 sub-tests of test-cases as shown in Table \ref{table:running_time}.  Clearly, our proposed algorithm has the best averages of running time in most of test-cases, excepted for test-case T$2$. However, the non-monotonic algorithm (Nm)~\cite{Kim2013} is unstable in other kinds of test-cases. Furthermore, the difference between our proposed algorithm and the non-monotonic algorithm for test-case T$2$  is not considerable. In comparison to the most stable algorithm Fast~\cite{Bro1997}, our algorithm has the similar accuracy of optimal values, but it runs from 9 times to 21 times faster than algorithm Fast~\cite{Bro1997}.
		
		\subsection{Optimality}
		
		\begin{table}
			\caption{Average of $|f(x_k) - f^*|$ in test-cases $T_1,...,T_6$}
			\label{table:optimality}
			\centering
			{\begin{tabular}{lccccc}
					\hline\noalign{\smallskip}
					Test-case & \textbf{Antilop} &  \textbf{Fast} &  \textbf{Nm} &  \textbf{Accer} &  \textbf{Anti+Acc} \\
					\noalign{\smallskip}\hline\noalign{\smallskip}
					T1 &  7E-15 &  \textbf{2E-15} &  2E+03 &  2E-03 &  2E-03 \\ 
					T2 &  \textbf{6E-08} &  7E-08 &  6E-08 &  7E-08 &  1E-07 \\ 
					T3 &  6E-16 &  \textbf{2E-16} &  2E-01 &  1E-01 &  2E-04 \\ 
					T4 &  9E-09 &  8E-09 &  6E+03 &  2E+03 &  \textbf{1E-10} \\ 
					T5 &  3E-09 &  \textbf{9E-10} &  5E+05 &  5E+05 &  1E+03 \\ 
					T6 &  6E-03 &  4E-03 &  7E+09 &  1E+09 &  \textbf{6E-04} \\ 
					\noalign{\smallskip}\hline
				\end{tabular}}
			\end{table}
			
			This section is dedicated to investigate the optimal effectiveness of the compared algorithms. The results shown in Table~\ref{table:optimality} are the average values $|f(x_k)-f^*|$ of 5 sub-test cases for 6 test-cases. Obviously, our algorithm and Fast~\cite{Bro1997} best values that are highly distinguished from the other methods. All the methods have errors $|f(x_k) - f^*|$ that can be raised by approximately computing float numbers in computers and characteristics of iterative methods. Hence, the minor differences between the proposed algorithm's results and the best results can be acceptable for the size of large matrix $d \times n$, and the big values of Lipschitz constant related to $||A||_2$.

			\section{Conclusion and Discussion}\label{sec:conclusion}
			
			In the paper, we proposed a fast robust anti-lopsided algorithm to solve the nonnegative least squares problem that is one of the most fundamental problem in data modeling. We theoretically proved that our algorithm has linear convergence $(1-\frac{1}{2||Q||_2})^k$ on the sub-space of passive variables, where $\sqrt{n} \leq ||Q||_2 \leq n$, and $n$ is the dimension of solutions. 
			
			In addition, we carefully compare the proposed algorithm with state-of-the-art algorithms in different major research directions on large matrices about three aspects: convergence rate, running time and optimality of solutions on 30 randomized tests among 6 different test-cases which often occur in practice. On convergence rate and running time, the proposed algorithm's results are highly distinct from other ones' results. On the optimality of solutions, the results of our algorithm are very close to the most accurate results. These differences are not significant and inevitable because they are raised by computing approximately float numbers in computers and the characteristics of iterative methods using many float operators on the large-scale matrices. Therefore, it can be said that the proposed algorithm obtains all of three significant aspects on convergence speed, running time and accuracy.
			
			Finally, the convergence speed of the Nesterov accelerated method~\cite{Guan2012} using the anti-lopsided steps is much more faster than Nesterov accelerated method directly applied. Hence, we strongly believe that  the anti-lopsided steps can have a significant impact on iterative methods using the first derivative for solving least squares problems and related quadratic programming problems.

\section{Acknowledgements}

This work was supported by 322 Scholarship from Vietnam Ministry of Education and Training; and Asian Office of Aerospace R\&D under agreement number FA2386-13-1-4046.

\bibliography{NNLSRef}   

\end{document}

%% file: defs.tex
\newcommand{\BA}{\begin{array}}
\newcommand{\EA}{\end{array}}

\newcommand{\argmin}{\mathop{\rm argmin}}

\newtheorem{theorem}{Theorem}[section]

\newtheorem{lemma}[theorem]{Lemma}

{\begin{quote}}{\end{quote}}

\makeatletter
\long\def\@makecaption#1#2{
   \vskip 9pt
   \begin{small}
   \setbox\@tempboxa\hbox{{\bf #1:} #2}
   \ifdim \wd\@tempboxa > 5.5in
        \begin{center}
        \begin{minipage}[t]{5.5in}
        \addtolength{\baselineskip}{-0.95pt}
        {\bf #1:} #2 \par
        \addtolength{\baselineskip}{0.95pt}
        \end{minipage}
        \end{center}
   \else
    \hbox to\hsize{\hfil\box\@tempboxa\hfil}
   \fi
   \end{small}\par
}
\makeatother

\newcounter{oursection}

\newcounter{lecture}